\documentclass[12pt]{amsart}

\usepackage{amssymb}
\usepackage{amsmath}
\usepackage{bm}
\usepackage{comment}
\numberwithin{equation}{section}

\newtheorem{theorem}{Theorem}

\newtheorem{proposition}{Proposition}

\numberwithin{theorem}{section}

\numberwithin{corollary}{section}
\numberwithin{lemma}{section}
\numberwithin{definition}{section}
\numberwithin{proposition}{section}
\numberwithin{remark}{section}

\def\F{{\bf{F}}}

\def\al{\alpha}

\def\re{\mathbb{R}}
\def\N{\mathbb N}

\def\R{\mathbb R}

\def\N{\mathbb{N}}

\def\pd{\partial}
\def\ol{\overline}

\def\la{\lambda}
\def\disp{\displaystyle}
\def\({\left(}
\def\){\right)}

\def\phi{\varphi}

\begin{document}
 
\title[Finsler Hardy-Kato's inequality]{Finsler Hardy-Kato's inequality }

\author[A. Alvino]{A. Alvino$^1$}
\author[A. Ferone]{A. Ferone$^2$}
\author[A. Mercaldo]{A. Mercaldo$^1$}
\author[F. Takahashi]{F. Takahashi$^3$}
\author[R. Volpicelli]{R. Volpicelli$^1$}

\date{\today}

\setcounter{footnote}{1}
\footnotetext{Universit\`a di Napoli Federico II, Dipartimento di Matematica e Applicazioni ``R. Caccioppoli'',
Complesso Monte S. Angelo, via Cintia, 80126 Napoli, Italy;\\
e-mail: {\tt angelo.alvino@unina.it,   mercaldo@unina.it, volpicelli@unina.it}}

\setcounter{footnote}{2}
\footnotetext{Universit\`a della Campania {\it Vanvitelli}, Dipartimento di Matematica e Fisica, via Lincoln 5, 81100 Caserta, Italy;\\
e-mail:{\tt adele.ferone@unicampania.it}}

\setcounter{footnote}{3}
\footnotetext{Department of Mathematics, Osaka City University,
3-3-138 Sugimoto, Sumiyoshi-ku, Osaka 558-8585, Japan. \\
e-mail:{\tt futoshi@sci.osaka-cu.ac.jp}}

\begin{abstract} 
\medskip
We prove  an improved version of the trace-Hardy inequality, so-called Kato's inequality, on the half-space in Finsler context.
The resulting inequality extends the former one obtained by \cite{AFV} in Euclidean context.
Also we discuss the validity of the same type of inequalities on open cones.

\vspace{1em}\noindent
\rm 

\vspace{1em}\noindent
\rm 
\end{abstract}
\maketitle

\bigskip

\section{Introduction}
\label{section:Introduction}

In the last decades interests in Finsler geometry have increased due to its possible applications in different contexts of mathematics, 
such as anisotropic eigenvalue problems and anisotropic evolution problems. 
One of the basic idea is to endow the space $\R^N$ with the distance obtained by a Finsler metric and to extend classical results 
to such a new geometrical context.

In this paper we are interested in the trace-Hardy inequality, so-called Hardy-Kato's inequality, on the half-space $\R^N_+=\R^N\times[0,+\infty)$ 
endowed with a Finsler norm.
That is Hardy inequality for Sobolev functions defined on $\R^N_+$ with non-zero trace on the boundary of $\R^N_+$ in Finsler context. 
More generally we also treat Hardy-Kato's inequality on open cones endowed with Finsler norm.

The interest in the theory of boundary trace for Sobolev functions and Hardy's inequalities arises from the possible applications 
to boundary value problems for PDEs and nonlinear analysis. 
They have been developed by various authors via different methods in different settings: 
Here we just recall some recent papers and references therein \cite{FMT}, \cite{AVV}, \cite{AFV}, \cite{VH}.

Let us begin by discussing the case of the half-space $\R^N_+$. 
In \cite{AFV} a sharp trace-Hardy inequality has been proved:
For any $2 \le \beta < N$ there exists a positive constant $K(N, \beta)$ such that
$\displaystyle\lim_{\beta \to N} K(N, \beta) = 0$ and
\begin{equation}
\label{HardyAFV}
	K(N,\beta)
	\int_{\pd {\re^N_{+}}}\frac {u^2}{|x|} \,dx + \frac {(\beta-2)^2}{4}\int_ {{\re^N_{+}}}\frac {u^2}{|x|^2+t^2} dx dt
	\le \int_ {{\re^N_{+}}} |\nabla u|^2 \, dx dt
\end{equation}
holds for any function $u$ in the Sobolev space $W^{1,2}(\R^{N}_+)$. 
The constant $K(N, \beta)$ is computed explicitly as
\begin {equation}
\label{constant}
	K(N,\beta) = 2\frac{\Gamma \( \frac{N+\beta}{4}-\frac{1}{2} \) \Gamma \( \frac{N-\beta}{4}+\frac{1}{2} \)}
		{\Gamma \( \frac{N+\beta}{4}-1 \) \Gamma \( \frac{N-\beta}{4} \)},
\end{equation} 
and both constants $K(N, \beta)$ and $\frac{(\beta-2)^2}{4}$ in \eqref{HardyAFV} are sharp. 
Inequality \eqref{HardyAFV} interpolates the classical Kato's inequality, which corresponds to $\beta=2$ in \eqref{HardyAFV}, 
and the Hardy inequality on $\re^N_{+}$ obtained by letting $\beta$ go to $N$.

As pointed out, our goal is to prove the trace-Hardy inequality \eqref{HardyAFV} in a more general geometrical framework. 
We consider $\R^{N}_{+}$ as the product space $\R^{N-1}\times \R_{+}$ 
and we endow it with a natural product metric generated by a Finsler norm $H^0$ on $\R^{N-1}$ and the usual Euclidian norm on $\R$. 
Denote each point $z \in \R^N_{+}$ as a couple $(x,t)$ where $x \in \R^{N-1}$, $t \in \R_{+}$, 
and consider the norm $\Phi^0$:
\begin{equation}
\label{Phi0}
	\Phi^0(z) = \Phi^0(x,t)=\sqrt{[H^0(x)]^2+t^2}, \quad \>\> z=(x,t)\in \R^{N-1}\times \R_{+}. 
\end{equation}
The {\sl dual norm} $\Phi$ of $\Phi^0$: 
\begin{equation}
\label{Phi}
	\Phi(\eta)=\Phi(\xi,t)=\sqrt{[H(\xi)]^2+t^2}, \quad \>\>\eta=(\xi,t) \in \R^{N-1} \times \R_{+}
\end{equation}
is automatically introduced to evaluate the length of the gradient of a function, 
where $H = H(\xi)$ denotes the dual norm of $H^0 = H^0(x)$ defined on $\R^{N-1}$.
We refer to \S \ref{section:notations} for the definitions, notations, and main properties of a Finsler norm.

%
%
Our first main result is the following: 
\begin{theorem}
\label{Theorem:main}
Assume $N \ge 3$. 
Let $H$ be a Finsler norm on $\R^{N-1}$ and let $\Phi(\xi,t)$ be the Finsler norm in $\R_{+}^N$ defined by \eqref{Phi}.
Then for any $u \in W^{1,2}(\re^N_{+})$ and $2 \le \beta < N$, 
\begin{equation}
\label{Finsler-Kato}
	 K(N,\beta)
	 \int_{\pd {\re^N_{+}}} \frac{u^2(x,0)}{\Phi^0(x,0)} dx
	 \le \int_{{\re^N_{+}}} \Phi^2(\nabla u) dz - \frac{(\beta-2)^2}{4} \int_{{\re^N_{+}}} \frac{u^2(z)}{[\Phi^0(z)]^2} dz
\end{equation}
holds where $\Phi^0$ is defined by \eqref{Phi0}, $\nabla u(z) = ( \nabla_x u, \frac{\pd u}{\pd t})$, $dz = dx dt$ for $z = (x, t) \in \re^N_{+}$,
and $K(N, \beta)$ is defined in \eqref{constant}.
$K(N,\beta)$ is sharp in the sense that
\begin{equation}
\label{optimal}
	K(N, \beta) 
	= \inf_{u \in W^{1,2}(\R^N_+), u \ne 0} 
	\frac{\displaystyle\int_{{\re^N_{+}}} \Phi^2(\nabla u) \, dz - \frac{(\beta-2)^2}{4} \displaystyle\int_{{\re^N_{+}}} \frac{u^2(z)}{[\Phi^0(z)]^2} dz}
	{\displaystyle\int_{\pd {\re^N_{+}}} \frac{u^2(x,0)}{\Phi^0(x,0)} dx}
\end{equation}
holds true.
\end{theorem}

The non-attainability of the optimal constant $K(N,\beta)$ can be seen as follows:
If the infimum \eqref{optimal} were attained by a function $u \in W^{1,2}(\R^N_{+})$, 
then it is proportional to the solution of the problem
\begin{equation}
\label{Eq:P}
	\begin{cases}
	\Delta_{\Phi} \phi(x,t) + \disp{\frac{(\beta-2)^2}{4} \frac{\phi(x,t)}{[\Phi^0(x,t)] ^2}} = 0 & \textrm{in} \>\> \re^N_{+},\\
	\phi(x,0) = [\Phi^0(x,0)] ^{-\frac{N}{2}+1} & \textrm {on} \>\> \pd \re^N_{+}.
	\end{cases}
\end{equation}
Here 
\[
	\Delta_\Phi = \Delta_{H,x} + \frac{\pd^2}{\pd t^2},
\]
and 
\[
	\Delta_{H,x} \phi 
	= \sum_{j=1}^{N-1} \frac{\pd}{\pd \xi_j} \( H(\nabla_x \phi)(H_{\xi_j}(\nabla_x \phi) \)
\]
is the so-called Finsler-Laplace operator with respect to the Finsler norm $H$ on $\R^{N-1}$ (see \S \ref{section:notations} for the definition). 
However, we see that solution $\phi$ to \eqref{Eq:P} satisfies $\phi \notin W^{1,2}(\R^N_{+})$.
Actually in \S \ref{section:extremal} we prove that the solutions to \eqref{Eq:P} are of the form
\[
	\phi(x,y) = \Phi^0(x,t)^{-\frac N 2 +1}\, w(\sin^2\theta), \quad \theta = \arctan \frac t {H^0(x)}\,,
\]
where the function $w$ is expressed in terms of the hyper-geometric series, i.e.
\begin{equation} 
\label{series}
	F(a,b,c;y)= 1 + \frac{\Gamma(c)}{\Gamma(a)\Gamma(b)}\sum_{k=1}^\infty \frac{\Gamma(a+k)\Gamma(b+k)}{\Gamma(c+k)}\frac{y^k}{k!}
\end{equation}
(see \eqref{w_form} in \S \ref{section:extremal}) and the optimal constant is given by
\begin{equation}
\label{Kw}
	K(N,\beta)= -\lim_{\theta\to 0} (\sin 2\theta) w'(\sin^2\theta).
\end{equation}

Theorem \ref{Theorem:main} is obtained by using a very classical method of Calculus of Variations introduced 
by Weierstrass and developed by Schwartz, Lichtenstein and Morrey (we refer to \cite{GH} for the general theory and references therein).   
It has been adopted in \cite{AFV} and \cite{F} to prove inequality \eqref{HardyAFV} and previously, 
in \cite{Al2} to find an improvement of the classical Sobolev inequality. 
It consists of proving that a solution of the Euler-Lagrange equation of a suitable functional is, actually, a minimum. 
Such method is the crucial tool of our approach, since we deal with functions having non zero trace on the boundary.
For more precise description of the method, we refer to \cite{AFV}.

Finally in \S \ref{section:cone} we face the case of open cones and we show that the same method can be applied to prove the Hardy-Kato inequality 
in the cone 
\begin{equation*}
	C_\alpha = \left\{ (x,t)\in \R^N_+: t> (\tan\alpha)\, H^0(x) \right\}, \quad -\frac{\pi}{2} < \alpha < \frac{\pi}{2}.
\end{equation*}
(See \cite{Na1}, \cite{Na2} and \cite{VH} for similar results). 
In the following,
\begin{equation}
\label{area}
	d\sigma_{x,t} = \sqrt{1 + (\tan^2 \alpha) |\nabla H^0(x)|^2} dx, \quad x \in \re^{N-1}
\end{equation}
denotes an $(N-1)$-dimensional surface measure on $\pd C_{\al}$.
%
%
\begin{theorem}
\label{Theorem:cone}
Assume $N \ge 3$ and $2 \le \beta < N$. 
Let $\Phi^0(x,t)$ and $\Phi(\xi,t)$ be Finsler norms on $\re^N_{+}$ defined by \eqref{Phi0} and \eqref{Phi} respectively.
Then there exists a constant $K(N, \alpha, \beta) \in \re$ such that 
\begin{align}
\label{Eq:Kato(cone)}
	K(N,\alpha, \beta) \sqrt{1 + \tan^2 \al} & \int_{\pd C_{\alpha}} \frac{u^2(x,t)}{\Phi^0(x,t)} \frac{d\sigma_{x,t}}{\sqrt{1 + (\tan^2 \al) |\nabla H^0(x)|^2}} \\
	 \le &\int_{C_{\alpha}} \Phi^2(\nabla u) \, dxdt - \frac {(\beta-2)^2}{4}\int_{C_{\alpha}} \frac{u^2(x,t)}{[\Phi^0(x,t)]^2} dxdt \notag
\end{align}
holds true for any $u \in W^{1,2}(C_{\alpha})$. 
The constant $K(N, \alpha, \beta)$ is given by
\[
	K(N,\alpha,\beta) = -(\sin 2\alpha) \frac{w^{\prime}(\sin^2 \alpha)}{A_{\al, \beta}},
\]
where $w$ is defined in \eqref{w_form} with $k = -K(N, \beta)$ for $K(N, \beta)$ in \eqref{constant},
and $A_{\al, \beta}$ is defined in \eqref{A_ab}.
\end{theorem}
A proof of Theorem \ref{Theorem:cone} is given in \S \ref{section:cone}.
Note that the left-hand side of \eqref{Eq:Kato(cone)} is written as
\begin{align*}
	 K(N,\alpha, \beta) \int_{\re^{N-1}} \frac{u^2(x,(\tan \al) H^0(x))}{H^0(x)} dx.
\end{align*}
Note also that by \eqref{Kw} and the fact $A_{\al, \beta} \Big |_{\al = 0} = 1$, we clearly observe that
$$
	\lim_{\alpha\to 0} K(N,\alpha,\beta) = K(N,\beta).
$$

%
%
\section{Notations and preliminary results}
\label{section:notations}

In this section, we introduce some notations. 
Let $n \in \N$ be an integer and
let $H: \R^n \to [0,+\infty)$ be a continuous function satisfying the following properties
\begin{align}
\label{homogeneous}
	&H(\lambda \xi) = |\lambda| H(\xi), \quad \forall \xi \in \R^n, \forall \lambda \in \R, \\
\label{bound}
	&\gamma_1 |\xi| \le H(\xi) \le \gamma_2 |\xi|, \quad \forall \xi \in \R^{n}
\end{align}
for two positive constants $0<\gamma_1 \le \gamma_2 < +\infty$.
We denote the unit $H$-ball as
$$
	B_H =\{ \xi \in \R^n \,: \, H(\xi) < 1 \}\,.
$$
The {\it dual function}, or {\it polar function}, $H^0: \R^n \to [0,+\infty)$ of $H$ is defined by the formula
$$
	H^0(x) = \sup_{\xi \in \R^n \setminus \{ 0 \}} \frac{\langle \xi, x \rangle_{n}}{H(\xi)} 
	= \sup_{\xi \in B_H} \,\langle \xi, x \rangle_{n} \, ,\quad x \in \R^{n},
$$
here and in the following,
$\langle \xi ,x \rangle_{n} = \sum_{j=1}^n \xi_j x_j$ denotes the Euclidean inner product of $\R^n$.  
Note that by definition, it holds that
\begin{equation}
\label{Schwarz}
	|\langle \xi, x \rangle_n| \le H(\xi) H^0(x), \quad \xi, x \in \R^n.
\end{equation}
It is known that $H^0$ is a convex, continuous function on $\R^n$,
which satisfies the following properties
\[
	H^0(\lambda x) = |\lambda| H^0(x), \quad \forall x \in \R^n, \forall \lambda \in \R.
\]
\begin{equation}
\label{bound0}
	\frac{1}{\gamma_2} |x| \le H^0(x) \le \frac{1}{\gamma_1} |x|, \quad \forall x \in \R^n.
\end{equation}
A function  $H \in C^2 \( \re^n \setminus \{0\} \)$ is a {\it Finsler  norm},
if it satisfies properties \eqref{homogeneous}, \eqref{bound}, and it is strictly convex. 
For references about Finsler norms (or, more generally, for Finsler metrics) see \cite{BCS}, \cite{BP}.

Here we just recall the following properties:
if $H$ is a Finsler norm, then $H$ is the polar function of $H^0$, that is the following equality holds true
$$
	H(\xi) = (H^0)^0(\xi) = \sup_{x \in \re^n \setminus \{ 0 \}} \frac{\langle \xi ,x \rangle_{n}}{H^0(x)},
$$
and $H^0$ is the {\it gauge function} of the closed convex set $\ol{B_H}$.
Moreover we have the following basic identities whose proof can be found, for example, in \cite{BP} Lemma 2.1, 2.2, or \cite{VS} Proposition 6.2.
\begin{align}
\label{(1)}
	\nabla H (\la \xi) = \frac{\la}{|\la|} \nabla H(\xi), &\quad \forall  \xi \in \R^{n} \setminus \{ 0\}, \forall \la \in \R \setminus \{ 0\}, \\
\label{(2)}
	\langle \nabla  H(\xi), \xi \rangle_{n}= H(\xi), &\quad \forall \xi \in \R^{n} \setminus \{ 0\}, \\
\label{(3)}
	H \( \nabla H^0(x) \) = 1, &\quad \forall x \in \R^{n} \setminus \{ 0\}, \\
\label{(4)}
	\nabla H \(\nabla H^0(x) \) =\frac{x}{H^0(x)}, &\quad \forall x \in \R^{n} \setminus \{ 0\}.
\end{align}
Analogous properties hold true for $H^0$ by taking into account that $H(\xi) = (H^0)^0(\xi)$.

Finally we recall that if $H: \R^n \to [0,+\infty)$ is a Finsler norm, 
the {\it Finsler-Laplace operator} $\Delta_H$ is defined as
\begin{align*}
	\Delta_H u (x) &= {\rm div} \( H(\nabla u) \nabla_{\xi} H(\nabla u) \)(x) \\
	&= \sum_{j=1}^{n} \frac{\pd}{\pd \xi_j} \( H(\xi) H_{\xi_j}(\xi) \) \Big |_{\xi = \nabla u(x)} 
\end{align*}
for any  function $u\in C^2(\R^{n})$.

%
%
\section{Construction of exstremals}
\label{section:extremal}

This section is devoted to the construction of a smooth solution to \eqref{Eq:P}.
Let $N \ge 3$.
We denote $\re_+ = [0,+\infty)$, $\re^N_{+} = \re^{N-1} \times \re_+$ and $z = (x,t) \in \re^N_{+}$. 
For a function $u=u(x,t)$ in the Sobolev space $W^{1,2}(\R^N_+)$, 
$\nabla u = (\nabla_x u,\frac{\pd u}{\pd t})$ denotes its full gradient where
$\nabla_x u = \(\frac{\partial u}{\partial x_1}, \cdots, \frac{\partial u}{\partial x_{N-1}}\)$. 

\begin{proposition}
\label{prop:extremal}
Let $2\le \beta < N$ and let $K(N,\beta)$ be the constant defined in \eqref{constant}. 
Then the functions
\begin{multline}
\label{extremal}
	\phi(x,t) 
= \frac{1}{\left[ \Phi^0(x,t) \right]^{\frac{N-2}{2}}} F\( \frac{N+\beta}{4}-1,\frac{N-\beta}{4},\frac{1}{2};\frac{t^2}{\left[ \Phi^0(x,t) \right]^2} \) \\
	 - t\frac{K(N,\beta)}{\left[ \Phi^0(x,t) \right]^{\frac{N}{2}}} 
	F\(\frac{N+\beta}{4}-\frac{1}{2},\frac{N-\beta}{4} + \frac{1}{2},\frac{3}{2};\frac{t^2}{\left[ \Phi^0(x,t) \right]^2}\)
\end{multline}
are regular solutions to the problem \eqref{Eq:P}.
Moreover, $\phi$ in \eqref{extremal} satisfies
\begin{equation}
\label{normal_derivative}
	\frac{\pd \phi}{\pd t}(x, 0) = - \frac{K(N, \beta)}{[\Phi^0(x,0)]^{\frac{N}{2}}}. 
\end{equation}
\end{proposition}

\begin{proof} [Proof of Proposition \ref{prop:extremal}]
Define new variables
\begin{equation}
\label{rho_theta}
\begin{cases}
	&\rho = \Phi^0(x,t)=\sqrt {[H^0(x)]^2+t^2},\\
	&\theta = {\rm arctan}\disp\frac{t}{\Phi^0(x,0)} = {\rm arctan}\disp\frac{t}{H^0(x)}, \quad 0<\theta<\frac \pi 2.
\end{cases}
\end{equation}
Then we have
\begin{align*}
	\frac{\pd \rho}{\pd t} = \frac{t}{\rho}, &\quad \frac{\pd \theta}{\pd t} = \frac{H^0(x)}{\rho^2}, \\
	\nabla_x \rho = \frac{H^0(x)}{\rho}\nabla H^0(x), &\quad \nabla_x \theta = -\frac{t}{\rho^2} \nabla_x H^0(x).
\end{align*}
Thus we see
\begin{align}
\label{phi_t}
 	&\phi_t = \frac{\phi_\rho}{ \rho}t+\frac {\phi_\theta}{ \rho^2}H^0(x), \\
\label{phi_tt}
 	&\phi_{tt} = \phi_{\rho\rho} \frac{t^2}{\rho^2}  + 2\phi_{\rho\theta} \frac{t H^0(x)}{\rho^3}
 	+ \phi_{\theta\theta} \frac{(H^0(x))^2}{\rho^4} \\
	&\hspace{4em} + \phi_{\rho} \( \frac{1}{\rho} - \frac{t^2}{\rho^3} \) - 2 \phi_{\theta} \frac{t H^0(x)}{\rho^4}, \notag \\
\label{gradx}
	&\nabla_x\phi(x,t) = \( \frac {\phi_\rho}{ \rho}-\tan \theta \frac {\phi_\theta}{ \rho^2} \) H^0(x) \nabla H^0(x).
\end{align} 
Moreover by \eqref{homogeneous}, \eqref{(1)}, \eqref{(3)}, \eqref{(4)} and \eqref{gradx}, we have
\begin{equation}
\label{Hgrad}
	H(\nabla_x\phi(x,t))\nabla H(\nabla_x\phi(x,t)) = \(\frac{\phi_\rho}{\rho}-\tan \theta \frac{\phi_\theta}{\rho^2}\) x.
\end{equation}
Thus by \eqref{(2)}, we have
\begin{align}
\label{Delta_H}
	\Delta_{H, x} \phi &= {\rm div}_x \( H(\nabla_x\phi(x,t))\nabla H(\nabla_x\phi(x,t)) \) \\
 	&= \phi_{\rho\rho} \frac{(H^0(x))^2}{\rho^2}  - 2\phi_{\rho\theta} \frac{t H^0(x)}{\rho^3}
 	+ \phi_{\theta\theta} \frac{t^2}{\rho^4}  \notag \\
	&+ \phi_{\rho} \( \frac{N-1}{\rho} - \frac{(H^0(x))^2}{\rho^3} \) + \phi_{\theta} \( \frac{2 t H^0(x)}{\rho^4} - \frac{N-2}{\rho^2} \tan \theta \). \notag
\end{align} 
Therefore by \eqref{phi_tt}, \eqref{Delta_H}, and the fact $\Delta_{\Phi} = \Delta_{H,x} + \frac{\pd^2}{\pd t^2}$,
the equation \eqref{Eq:P} in the new variables \eqref{rho_theta} can be written as 
\begin{equation}
\label{Eq:polar}
	\phi_{\rho\rho} + (N-1)\frac{\phi_\rho}{\rho} - (N-2)\frac{\phi_\theta}{\rho^2} \tan \theta + \frac{\phi_{\theta\theta}}{\rho^2}
	= -\( \frac{\beta-2}{2} \)^2 \frac{\phi}{\rho^2}.
\end{equation}
Searching for solutions to \eqref{Eq:polar} of the form
\begin{equation}
\label{phi_form}
	\phi(x,t)=\rho^{-\frac {N}{2}+1}f(\theta)\,, 
\end{equation}
we see that the problem  \eqref{Eq:P} is equivalent to the following limit problem:
\begin{equation}
\label{Eq:f}
\begin{cases}
	f''(\theta)-(N-2) (\tan \theta) f'(\theta)-\(\frac {(N-2)^2}{4} - \frac {(\beta-2)^2}{4}\)f(\theta) = 0 & \theta \in (0,\frac{\pi}{2}), \\
	f(0) = 1, \quad \disp{\lim_{\theta \to \frac{\pi}{2}} f(\theta) \in \R}.
\end{cases}
\end{equation}
Problem \eqref{Eq:f} is explicitly solved in \cite{PZ} (pp. 271, eq.131) (see also \cite{AFV}). 
Indeed, $f(\theta) = w(\sin^2\theta)$, and $w$ is given by
\begin{equation}
\label{w_form}
	w(y) = F \(\frac{N+\beta}{4}-1,\frac{N-\beta}{4},\frac{1}{2};y\) 
	+ k \sqrt y F \(\frac{N+\beta}{4}-\frac{1}{2},\frac{N-\beta}{4}+\frac{1}{2},\frac{3}{2}; y \) 
\end{equation}
for a suitable constant $k$.
Here $F(a,b,c;y)$ is the hypergeometric series given in \eqref{series} which is convergent for $0 \le y < 1$.
Moreover in \cite{AFV}, it is proved that $f$ is a bounded solution to $(\ref{Eq:f})$, i.e.
$\lim_{y \to 1} w(y) \in \re$ holds true, if and only if $k = -K(N,\beta)$.  
Here we repeat those arguments for the sake of completeness, analyzing the behavior of a hypergeometric function near the point $y=1$. 
For this purpose,  we recall that (see \cite {AS} pp. 559) 
\begin{equation}
\label{formula2}
	\lim_{y \to 1} \frac {F(a,b,c;y)}{\ln(1-y)}=-\frac {\Gamma(a+b)}{\Gamma (a)\Gamma(b)}, \quad if \quad  c-a-b=0,
\end{equation}
\begin{equation}
\label{formula3}
	\lim_{y \to 1} \frac {F(a,b,c;y)}{(1-y)^{c-a-b}}=\frac {\Gamma (c)\Gamma(a+b-c)}{\Gamma (a)\Gamma(b)}, \quad if \quad  c-a-b<0.
\end{equation}
An easy calculation shows that for both hypergeometric functions appearing in (\ref{w_form}), $c-a-b = \frac{3-N}{2} \le 0$ when $N \ge 3$. 
Let us first examine the case $N > 3$.
We write
\begin{align*}
	&\lim_{y \to 1} w(y) \\
	&=\lim_{y \to 1} (1-y)^{\frac {3-N}{2}}
	\left [ \frac{F\(\frac{N+\beta}{4}-1,\frac{N-\beta}{4},\frac{1}{2};y\)}{(1-y)^{\frac {3-N}{2}}}
	+ k \sqrt y \frac{F\(\frac{N+\beta}{4}-\frac {1}{2},\frac{N-\beta}{4}+\frac {1}{2},\frac{3}{2};y\)}{(1-y)^{\frac {3-N}{2}}}\right ].
\end{align*}
By \eqref{formula3}, the formula (\cite{AS}, pp. 557)
\begin{equation}
\label{formula_AS}
	\frac{d}{dy}F(a,b,c;y) = \,\frac{ab}{c} \,F(a+1,b+1,c+1;y),
\end{equation}
and de l'Hopital Theorem, we have that the limit is finite if and only if
\begin{equation}
\label{help}
	\frac{\Gamma \(\frac{1}{2} \) \Gamma \(\frac {N-3}{2}\)}{\Gamma \(\frac{N-\beta}{4}-1\) \Gamma \(\frac{N+\beta}{4}\)}
	+ k \frac{\Gamma \(\frac{3}{2}\) \Gamma \(\frac {N-3}{2}\)}{\Gamma \(\frac{N+\beta}{4} -\frac{1}{2}\) \Gamma \(\frac{N-\beta}{4} + \frac{1}{2} \)} = 0.
\end{equation}
Since $\Gamma(\frac{1}{2})= 2 \Gamma(\frac{3}{2})$, therefore $k = -K(N,\beta)$ for $K(N,\beta)$ in \eqref{constant}.
The case $N=3$ follows in a similar way, by using \eqref{formula2} instead of \eqref{formula3}.
Putting $k = -K(N,\beta)$ in \eqref{w_form} and taking \eqref{Eq:polar}, \eqref{phi_form} into account, 
we deduce that
\[
	\phi (x,t) = \( \Phi^0(x,t) \)^{-\frac N 2 +1} w\(\frac {t^2}{[{\Phi^0(x,t)}]^2} \)
\]
is a bounded solution to the equation in \eqref{Eq:P}.
Since $F(a,b,c;0)=1$, we have $w(0) = 1$ and that $\phi(x,0) = [\Phi^0(x,0)]^{-\frac{N}{2} +1}$.
Also by \eqref{phi_t} and \eqref{phi_form}, we see $\phi_t(x,0) = f'(0)[\Phi^0(x,0)]^{-\frac{N}{2}}$.  
Let us evaluate  $f'(0)$.
To do this recall that $f(\theta)=w(\sin^2\theta)$ and $w$ is given by \eqref{w_form} with $k=-K(N,\beta)$. 
Taking the formula \eqref{formula_AS} into account and $F(a,b,c;0)=1$, it is easy to check that $f'(0)=-K(N,\beta)$. 
Thus we have \eqref{Kw} and $\frac{\pd \phi}{\pd t}(x,0) = -K(N, \beta) [\Phi^0(x,0)]^{-\frac{N}{2}}$. 
This completes the proof of Proposition \ref{prop:extremal}.
\end{proof}

%
%
\section{Proof of Theorem \ref{Theorem:main}}
\label{section:proof}

In this section, we prove Theorem \ref{Theorem:main}.
Let $\Phi$ be the Finsler norm in $\R^N_+$ defined in \eqref{Phi} and let $\phi$ be the solution to the problem \eqref{Eq:P} defined in (\ref{extremal}). 
As stated in  \S \ref{section:Introduction}, we follow the arguments in \cite{AFV}, \cite{F}, 
while some modification is needed to apply them in the general Finsler context. 

Define a vector field ${\bf F}:\re^N_{+}\times \R \ni (z, h) \mapsto {\bf F}(z,h) \in \R^{N+1}$ where $z = (x,t) \in \re^N_{+}$ as
\begin{multline}
\label{F}
	{\bf F}(z,h) 
	\equiv \Bigg(\frac{2h}{\phi(z)}\Phi(\nabla\phi)\nabla\Phi(\nabla \phi)\,,\frac{h^2}{\phi^2(z)}\Phi^2(\nabla\phi)+\frac{(\beta-2)^2}4 \frac {h^2}{[\Phi^0(z)]^2}\Bigg )\,\\
	= \Bigg(\frac{2h}{\phi(x,t)}H(\nabla_x\phi)\nabla H(\nabla_x\phi)\,,\frac{2h}{\phi(x,t)} \phi_t\,,
	\frac{h^2}{\phi^2(x,t)}[H^2(\nabla_x\phi)+(\phi_t)^2]+\frac{(\beta-2)^2}4 \frac {h^2}{[H^0(x)]^2+t^2}\Bigg ).
\end{multline} 
Direct calculation shows that ${\bf F}$ is divergence free. 
Indeed, 
\[
	\langle \nabla\Phi(\nabla\phi),\nabla \phi \rangle_N = \Phi(\nabla \phi)
\]
by \eqref{(2)} 
and recalling that $\phi$ satisfies
\[
	\Delta_{\Phi} \phi + \frac{(\beta-2)^2 }4 \frac {\phi}{[\Phi^0]^2} = 0, \quad z = (x, t) \in \R^N_{+},
\]
we have
\begin{multline*}
	{\rm div}_{z,h}{\bf F}
	= {\frac{2h}{\phi}}{\rm div}_z\(\Phi(\nabla\phi)\nabla\Phi(\nabla\phi)\)
	-{\frac{2h}{\phi^2}}\Phi(\nabla\phi)\langle\nabla\Phi(\nabla\phi),\nabla \phi\rangle_N\\
	+{\frac{2h}{\phi^2}}\Phi^2(\nabla\phi)
	+{\frac{2h}{\phi}} {\frac {(\beta-2)^2}{4}} {\frac {\phi}{[\Phi^0(z)]^2}}
	={\frac{2h}{\phi}} \left[\Delta_{\Phi} \phi + \frac{(\beta-2)^2 }4 \frac {\phi}{[\Phi^0(z)]^2}\right] = 0.
\end{multline*}

For every $r>0$, denote
$$
	B_{\Phi^0}(r)=\{z \in \re^N_{+}:\, \Phi^0(z) < r \}.
$$ 
Let $R>0$ and let $u\in C_0^\infty (\re^N)$ be a nonnegative function compactly supported on $B_{\Phi^0}(R)$. 
Denote by $\Omega$ the set of $\re^N_{+}\times \R_+$ given by the subgraph of $u$ which is projected into $B_{\Phi^0}(R)\setminus B_{\Phi^0}(r)$, for some $0<r<R$. 
We get that the flow of ${\bf F}$ across $\partial\Omega$ is zero, since ${\bf F}$ is divergence free. 
This means that, if $\nu$ is the unit outer normal to $\partial \Omega$, we have
\begin{equation}
\label{div0}
	\int_{\partial \Omega}\langle {\bf F}(z,h)\,,\nu\rangle_{N+1}d\mathcal H^N=0\,.
\end{equation}
Let us write explicitly the left hand side of \eqref{div0}.
Note that $\partial\Omega$ consists of the union of the following $N$-dimensional surfaces
\begin{align}
\label{s1}
	&\Sigma_1=\left\{(z,0)\in \re^N_{+} \times \R : z \in B_{\Phi^0(R)} \setminus B_{\Phi^0(r)} \right\}, \\
\label{s2}
	&\Sigma_2=\left\{(z,h) \in \partial\re^N_{+} \times \R : r < \Phi^0(z)< R, \, 0\le  h\le u(z) \right\}, \\
\label{s3}
	&\Sigma_3=\left\{(z,h) \in \re^N_{+} \times \R : \Phi^0(z)= r, \, 0 \le  h \le u(z) \right\}, \\
\label{s4}
	&\Sigma_4=\left\{(z,h) \in \re^N_{+} \times \R : z \in B_{\Phi^0(R)}\setminus B_{\Phi^0(r)}, \, h= u(z) \right\}.
\end{align}
If $\nu_i$ denotes the outer unit normal of $\Sigma_i$ $(i=1,2,3,4)$ with respect to the Euclidean norm, 
from \eqref{div0} we get
\begin{equation}
\label{div_sum}
	\sum_{i=1}^4 \int_{\Sigma_i}\langle {\bf F}\,,\nu_i\rangle_{N+1} d\mathcal H^N = 0.
\end{equation}
%
%
Since ${\bf F}(z,0)\equiv 0$, we have
\begin{equation}
\label{div_s1}
	\int_{\Sigma_1}\langle {\bf F}\,,\nu_1\rangle_{N+1} d\mathcal H^N = 0.
\end{equation}
%
%
As regards the flow across $\Sigma_2$, observe that $\nu_2\equiv -{\bf e}_N$, 
where ${\bf e}_N = (0, \cdots, 0, 1, 0)$ is the unit vector of the standard Euclidean basis of $\R^{N+1}$. 
Note that $d{\mathcal H}^N = dx dh$ on $\Sigma_2$.
Since $z = (x, 0) \in \partial\re^N_{+}$ and by definition \eqref{F} of $\bf F$, we get
\begin{align*}
	\int_{\Sigma_2}\langle \F, \nu_2 \rangle_{N+1} d{\mathcal H}^N &= -2\int_{r<\Phi^0(x,0)<R} dx \int_0^{u(x,0)} \frac{h}{\phi(x,0)} \phi_t(x,0) dh\notag\\
	&=-\int_{r<\Phi^0(x,0)<R} \frac{u^2(x,0) }{\phi(x,0)}{\phi_t(x,0)} dx.
\end{align*}
Since $\phi$ satisfies \eqref{Eq:P} and \eqref{normal_derivative}, we then obtain
\begin{equation}
\label{div_s2}
	\int_{\Sigma_2}\langle \F, \nu_2 \rangle_{N+1} d{\mathcal H}^N = K(N,\beta) \int_{r<\Phi^0(x,0)<R} \frac{u^2(x,0)}{\Phi^0(x,0)} dx.
\end{equation}

%
%
Let us now evaluate the flow across $\Sigma_3$. 
The unit normal to $\Sigma_3$ is given by $\nu_3 = \(-\frac{\nabla \Phi^0}{|\nabla \Phi^0|},0 \)$, 
so that by \eqref{F} we deduce
\begin{equation}
\label{div_s3}
	\int_{\Sigma_3}\langle \F, \nu_3 \rangle_{N+1} d{\mathcal H}^N = -\frac{1}{|\nabla\Phi^0|} \int_{\Sigma_3} \frac{2h}{\phi}
	\left\langle \Phi \( \nabla\phi \) \nabla\Phi \(\nabla \phi \)\,, \nabla \Phi^0 \right\rangle_N d{\mathcal H}^N.
\end{equation}
By \eqref{Phi}, \eqref{phi_t}, \eqref{Hgrad}, and \eqref{phi_form}, it follows
\begin{multline}
\label{Eq:Phi}
	 \left\langle \Phi \(\nabla\phi\) \nabla \Phi \(\nabla \phi \), {\nabla\Phi^0} \right\rangle_N 
	= \left\langle H(\nabla_x\phi)\,\nabla H(\nabla_x\phi)\,, \frac{H^0(x)\nabla H^0(x)}{\Phi^0(x,t)} \right \rangle_{N-1} 
	+ \frac{t \phi_t}{\Phi^0(x,t)} \\
	 = \rho^{-\frac{N}{2}-2} \Bigg [ \(\( -\frac{N}{2} +1 \)f(\theta) - \tan \theta f'(\theta )\) H^0(x)\, \langle x\,, \nabla H^0(x) \rangle_{N-1} \\
	 +t\(\(-\frac{N}{2} +1 \)f(\theta)t+f'(\theta )H^0(x)\) \Bigg ]
	 =\rho^{-\frac{N}{2}} \left[\( -\frac{N}{2} +1\) f(\theta) \right] = -\frac {N-2}{2\rho} \phi\,.
\end{multline}
Note that by \eqref{bound} and \eqref{(3)} we have
\[
	|\nabla\Phi^0| \ge \frac{1}{\gamma_2} \Phi \(\nabla\Phi^0 \) = \frac{1}{\gamma_2}\,.
\]
Thus collecting \eqref{div_s3} and \eqref{Eq:Phi}, we deduce
\begin{multline}
\label{div_s3_estimate}
	\left | \int_{\Sigma_3}\langle \F, \nu_3 \rangle_{N+1}\,d{\mathcal H}^N\right |
	\le \gamma_2\frac{(N-2)}{r}\int_{\Sigma_3}h\,d{\mathcal H}^N=\gamma_2\frac {N-2}r \int_{\Phi^0(z)=r}d{\mathcal H}^{N-1}\int_0^{u(z)}hdh \\
	\le \gamma_2\frac{(N-2)}{r}\,{\mathcal H}^{N-1}\left (\{z\in \R^{N-1}\times\R_+: \Phi^0(z)=r\}\right )\sup_{\Phi^0(z)=r}\frac {u^2(z)}2 \\
	\le \gamma_2\frac{N-2}{r} \, \frac {N\,k_N \,r^{N-1}}2\sup_{\Phi^0(z)=r}\frac {u^2(z)}2=O(r^{N-2})\,,
\end{multline}
where $k_N$ is the measure of $B_{\Phi^0}$.

%
%
It remains to estimates the flow of ${\bf F}$ across $\Sigma_4$. 
In such a case the normal $\nu_4$ is given by
$$
	\nu_4 = \frac{1}{\sqrt{|\nabla u|^2+1}} \(-\nabla u, 1 \) \in \re^{N+1},
$$
and then by \eqref{F}, it follows
\begin{multline}
\label{div_s4}
	\int_{\Sigma_4}\langle {\bf F}\,, \nu_4 \rangle_{N+1}d{\mathcal H}^N= \int_{r<\Phi^0(z)<R} \langle {\bf F}(z,u(z)), (-{\bf\nabla} u(z), 1)\rangle_{N+1}dz \\
	=\int_{r<\Phi^0(z)<R} \Bigg (-\frac{2u}{\phi} \Phi \(\nabla \phi \)\langle \nabla \Phi(\nabla \phi)\,, \nabla u \rangle_N
	+ \frac {u^2}{\phi^2}\Phi(\nabla \phi)^2+\frac{(\beta-2)^2}{4} \frac{u^2(z)}{[\Phi^0(z)]^2}\Bigg )dz \,.
\end{multline}
Here note that $d{\mathcal H}^N = \sqrt{1 + |\nabla u|^2} dz$ on $\Sigma_4$.
By convexity of $\Phi$, we get that
$$
	\Phi(\nabla u) \ge \Phi(\nabla \phi) + \langle \nabla\Phi(\nabla\phi)\,,\nabla u - \nabla\phi \rangle_N \,,
$$
and by \eqref{(2)}, \eqref{Schwarz}, and Young's inequality, we obtain
\begin{multline}
\label{convex}
	\frac{2u}{\phi}\Phi\(\nabla \phi\)\langle \nabla \Phi(\nabla \phi)\,, \nabla u\rangle_N\le \frac{2u}{\phi}\Phi\(\nabla \phi\) \left [ \Phi(\nabla u)-\Phi(\nabla \phi)+
	\langle \nabla\Phi(\nabla\phi)\,,\nabla\phi\rangle_N \right ]\\
	=\frac{2u}{\phi}\Phi \(\nabla \phi \)\, \Phi(\nabla u)\le \frac {u^2}{\phi^2}\Phi\( \nabla \phi \)^2+\Phi(\nabla u)^2\,.
\end{multline}
Finally, collecting \eqref{div_s4} and \eqref{convex} we deduce
\begin{equation}
\label{div_s4_estimate}
	- \int_{\Sigma_4}\langle {\bf F}\,, \nu_4\rangle_{N+1}d{\mathcal H}^N \le \int_{r<\Phi^0(z)<R} \( \Phi(\nabla u)^2  - \frac {(\beta-2)^2}4 \frac {u^2(z)}{[\Phi^0(z)]^2} \)dz\,.
\end{equation}
Collecting \eqref{div_sum}, \eqref{div_s1}, \eqref{div_s2}, \eqref{div_s3_estimate}, and \eqref{div_s4_estimate}, 
we obtain
\begin{multline}
\label{last}
	\int_{r<\Phi^0(z)<R} \( \Phi(\nabla u)^2 - \frac {(\beta-2)^2}4 \frac {u^2(z)}{[\Phi^0(z)]^2} \)dz\\
	\ge K(N,\beta) \int_{r<\Phi^0(x,0)<R} \frac{u^2(x,0) }{H^0(x)} dx+ O(r^{N-2})\,.
\end{multline}
Letting $r$ go to zero and $R$ go to infinity, we prove the inequality \eqref{Finsler-Kato}.

%
%
To prove the optimality of the constant that appears in (\ref{constant}),
repeat all the previous arguments on replacing $u$ by $\phi$. 
In such a case both inequalities \eqref{convex} and \eqref{div_s4_estimate} hold as equality. 
Moreover, since $\phi $ is not compactly supported in $B_{\Phi^0}(R)$, the extra $N$-dimensional surface has to be considered 
\begin{equation*}
	\Sigma_5=\left \{(z,h)\in\re^N_{+}\times\R : \Phi^0(z)= R\,, \quad 0\le  h\le \phi(z) \right \}\,.
\end{equation*}
The unit normal $\nu_5$ is given by $\nu_5=-\nu_3$,
so that, by \eqref{div_s3} and \eqref{Eq:Phi}, instead of \eqref{last}, we obtain
\begin{multline}
\label{last2}
	\int_{r<\Phi^0(z)<R} \( \Phi^2(\nabla \phi) -\frac{(\beta-2)^2}{4} \frac{\phi^2(z)}{[\Phi^0(x,t)]^2} \)dz \\
	= K(N,\beta) \int_{r<\Phi^0(x,0)<R} \frac{\phi^2(x,0)}{H^0(x)}dx - \frac{1}{|\nabla \Phi^0|} \frac {N-2}{2r}\int_{\Sigma_3} h d{\mathcal H}^N
	+ \frac{1}{|\nabla \Phi^0|}\frac {N-2}{2R} \int_{\Sigma_5} h d{\mathcal H}^N \,.
\end{multline}
It is easy to check that the last two integrals in (\ref{last2}) are equal. 
Indeed, by spherical coordinates, if $B_{\Phi^0}^+=B_{\Phi^0}\cap \R_+$, recalling \eqref{phi_form}, we have
\begin{align}
\label{uni}
	\frac 1 r \int_{\Sigma_3} h d{\mathcal H}^N 
	&=\frac 1 r \int_{\partial B_{\Phi^0}^+}r^{N-1}\( \int_0^{\phi(rx', rt')}h dh \)d{\mathcal H}^{N-1} \notag \\
	&=\frac{r^{N-1}}{2r} \int_{\partial B_{\Phi^0}^+}\phi^2(rx',rt')d{\mathcal H}^{N-1}  \notag \\ 
	&= \frac{r^{N-2}}2 \int_{\partial B_{\Phi^0}^+}r^{-N+2}f^2(\theta)d{\mathcal H}^{N-1} =\frac 1 R \int_{\Sigma_5} h d{\mathcal H}^N.
\end{align}
Collecting (\ref{last2}) and (\ref{uni}) we deduce 
\[
	\lim_{R \to \infty}\,\lim_{r \to  0} \frac{\int_{r<\Phi^0(z)<R} \Phi^2(\nabla \phi)(z) dz}{\int_{r<\Phi^0(x,0)<R} \frac{\phi^2(x,0)}{H^0(x)} dx} = K(N,\beta),
\]
which shows the optimality of the constant.
\qed

%
%
\section{Finsler Hardy-Kato's inequality in cones}
\label{section:cone}

In this section, we give a proof of Theorem \ref{Theorem:cone}. 
Let us consider the following open cone 
\[
	C_{\alpha} = \{(x,t) \in \re^{N-1} \times \re \,:\, t > (\tan \alpha) H^0(x) \}
\]
in $\re^N$ for some $\alpha\in (-\frac{\pi}{2}, \frac{\pi}{2})$.
Note that the unit outer normal vector on the $(N-1)$-dimensional surface
\[
	\pd C_{\alpha} = \{(x,t) \in \re^{N-1} \times \re \,:\, t = (\tan \alpha) H^0(x) \}
\]
is given by
\begin{equation}
\label{nu}
	\nu^\alpha(x, t) = \frac{1}{\sqrt{1 + \tan^2 \alpha |\nabla H^0(x)|^2}} ( (\tan \alpha) \nabla H^0(x), -1) \in \re^{N-1} \times \re,
\end{equation}
and the area element $d\sigma_{x,t}$ on $\pd C_{\al}$ is defined by \eqref{area}.

We repeat the same arguments used in \S \ref{section:extremal} and we look for solutions $\phi_{\al, \beta}$ to the problem
\begin{equation}
\label{Eq:cone}
	\begin{cases}
	\Delta_{\Phi} \phi_{\alpha, \beta}(x,t) + \disp{\frac{(\beta-2)^2}{4} \frac{\phi_{\alpha, \beta}(x,t)}{[\Phi^0(x,t)] ^2}} = 0 
	& \textrm{in} \, C_{\alpha},\\
	\phi_{\alpha, \beta} = [\Phi^0(x,t)] ^{-\frac{N}{2}+1} & \textrm {on} \, \pd C_{\alpha}
	\end{cases}
\end{equation}
of the form
\[
	\phi_{\alpha, \beta}(x,t) = \frac{1}{\( (H^0(x))^2 + t^2 \)^{\frac{N-2}{4}}}w_{\alpha, \beta} \( \frac{t^2}{(H^0(x))^2 + t^2} \),
\]
where $\Phi^0$, $\Phi$ are defined by \eqref{Phi0}, \eqref{Phi} respectively.
Then $w_{\alpha, \beta}(\sin^2 \theta) = g_{\al, \beta}(\theta)$ and $g_{\alpha, \beta}$ solves the problem
\[
	\begin{cases}
	g''(\theta)-(N-2) (\tan \theta) g'(\theta)- \(\frac {(N-2)^2}{4}- \frac{(\beta-2)^2}{4} \) g(\theta) = 0 & \theta \in (\alpha,\frac{\pi}{2}), \\
	g(\alpha) = 1, \quad \lim_{\theta \to \frac{\pi}{2}} g(\theta) \in \re.
	\end{cases}
\]
Thus $w_{\al, \beta}$ is described by using the hypergeometric function
\begin{equation}
\label{w_ab}
	w_{\alpha, \beta}(y) = c_1 F \(\frac{N+\beta}{4}-1,\frac{N-\beta}{4},\frac{1}{2};y \) 
	+ c_2 \sqrt{y} F \(\frac{N+\beta}{4} - \frac {1}{2}, \frac{N-\beta}{4}+\frac{1}{2},\frac{3}{2};y \)
\end{equation}
for suitable choice of constants $c_1$ and $c_2$.
The constants $c_1,c_2$ have to satisfy 
\begin{equation}
\label{limw1}
	w_{\al,\beta}(\sin^2\alpha)= 1, \quad  \lim_{\theta \to \frac{\pi}{2}}w_{\al, \beta}(\sin^2\theta) \in \R.
\end{equation}
By the first condition in \eqref{limw1}, we get
\begin{multline}
\label{limw2}
	c_1 F \( \frac{N+\beta}{4}-1,\frac{N-\beta}{4},\frac{1}{2};\sin^2\alpha \) \\
	+ c_2 |\sin\alpha| F \( \frac{N+\beta}{4}-\frac {1}{2},\frac{N-\beta}{4} + \frac{1}{2},\frac{3}{2};\sin^2\alpha \)=1.
\end{multline}
For the second condition in \eqref{limw1}, we have
\begin{equation}
\label{c22}
	c_1 \frac{\Gamma \( \frac{1}{2} \) \Gamma \(\frac{N-3}{2} \)}{\Gamma \( \frac{N-\beta}{4} -1 \) \Gamma \(\frac{N+\beta}{4} \)}
	+ c_2 \frac{\Gamma \( \frac{3}{2} \) \Gamma \(\frac{N-3}{2} \)}{\Gamma \( \frac{N+\beta}{4} -\frac{1}{2}\) \Gamma \(\frac{N-\beta}{4} + \frac{1}{2} \)}=0.
\end{equation}
Collecting \eqref{limw2} and \eqref{c22}, if we put
\begin{multline}
\label{A_ab}
	A_{\alpha, \beta} = F\(\frac{N+\beta}{4}-1,\frac{N-\beta}{4},\frac{1}{2};\sin^2\alpha \) \\
	- K(N,\beta)|\sin\alpha| F\(\frac{N+\beta}{4}-\frac {1}{2},\frac{N-\beta}{4}+\frac {1}{2},\frac{3}{2};\sin^2\alpha \),
\end{multline}
we get
\begin{equation}
\label{c1c2}
	c_1 = \frac{1}{A_{\alpha, \beta}}, \quad c_2 = -\frac{K(N,\beta)}{A_{\alpha, \beta}}.
\end{equation}
Therefore, if $w$ is as in \eqref{w_form} with $k = -K(N, \beta)$ and $\phi$ is defined in \eqref{extremal}, 
we obtain $w_{\alpha, \beta}(y) = \frac{1}{A_{\alpha, \beta}} w(y)$ and that
\begin{equation}
\label{phi_ab}
	\phi_{\alpha, \beta}(x,t) = \frac 1{ A_{\alpha, \beta}} \phi(x,t)
\end{equation}
is a solution to \eqref{Eq:cone}.

Let us check what happens in the proof of Theorem \ref{Theorem:main} when we work on $C_{\alpha}$. 
We start by defining the vector field $\F_{\al, \beta}$ by replacing $\phi$ with $\phi_{\al,\beta}$ in the definition \eqref{F}.
By \eqref{phi_ab}, we have $\F_{\al, \beta} \equiv \F$ where $\F$ is defined by using $\phi$ in \eqref{extremal}.
Also, instead of the surface defined in \eqref{s1}-\eqref{s4} we deal with the following $N$-dimensional hypersurfaces in $\re^{N+1}$:
\begin{align*}
	&\Sigma_{C_\alpha,1}=\left \{(z,0)\in C_{\alpha} \times \R : z\in B_{\Phi^0}(R) \setminus B_{\Phi^0}(r) \right \}, \\
	&\Sigma_{C_\alpha,2}=\left \{((x, t), h) \in \partial C_{\alpha} \times \R : r < \Phi^0(x,t)< R, \, 0 \le  h \le u(x,t), \, t = (\tan \al) H^0(x) \right \}, \\
	&\Sigma_{C_\alpha,3}=\left \{(z,h)\in C_{\alpha} \times \R : \Phi^0(z)= r, \, 0 \le  h\le u(z) \right \}, \\
	&\Sigma_{C_\alpha,4}=\left \{(z,h)\in C_{\alpha} \times \R : z\in B_{\Phi^0}(R) \setminus B_{\Phi^0}(r), \, h= u(z) \right \}.
\end{align*}
The unit outer normal vector on $\Sigma_{C_\alpha,2}$ is given by $\nu_2(x, t) =(\nu^\alpha (x,t),0)\in \R^{N+1}$ where $\nu^{\al}$ is defined in \eqref{nu}.
Thus
\begin{align*}
	&\langle \F_{\al, \beta}, \nu_2 \rangle_{N+1} \Big|_{(x, t, h) \in \Sigma_{C_\al,2}} 
	= \langle \( \frac{2h}{\phi_{\al, \beta}} H(\nabla_x\phi_{\al, \beta})(\nabla H)(\nabla_x\phi_{\al, \beta}), \, \frac{2h}{\phi_{\al, \beta}} (\phi_{\al, \beta})_t \), \, \nu^{\al} \rangle_N \Big|_{(x,t) \in \pd C_{\al}} \\
	&= \frac{1}{\sqrt{1 + (\tan^2 \alpha) |\nabla H^0(x)|^2}} \frac{2h}{\phi_{\al, \beta}} \left\{ (\tan \alpha) H(\nabla_x\phi_{\al, \beta})(\nabla H)(\nabla_x\phi) \cdot \nabla H^0(x) - (\phi_{\al, \beta})_t \right\}.
\end{align*}
Note that $d{\mathcal H}^N = d\sigma_{x,t} dh = \sqrt{1 + (\tan^2 \alpha) |\nabla H^0(x)|^2} dx dh$ on $\Sigma_{C_\alpha,2}$.
Thus noting the cancellation of the term $\sqrt{1 + (\tan^2 \alpha) |\nabla H^0(x)|^2}$, we see
\begin{align*}
	&\int_{\Sigma_{C_{\alpha},2}}\langle \F_{\al, \beta}, \nu_2 \rangle_{N+1} d{\mathcal H}^N \\ 
	&= \int_0^{u(x, (\tan \al)H^0(x))} 2h \times \\
	&\int_{\{ x \, : \, r<\Phi^0(x, (\tan \al)H^0(x))<R \}} 
	\frac{1}{\phi_{\al, \beta}} \left\{ (\tan \alpha) H(\nabla_x\phi_{\al, \beta})(\nabla H)(\nabla_x\phi_{\al, \beta}) \cdot \nabla H^0(x) - (\phi_{\al, \beta})_t \right\} dxdh \\
	&= \int_{\{ x \, : \, r<\Phi^0(x,(\tan \al) H^0(x))<R \}} \frac{u^2(x,(\tan \al) H^0(x))}{\phi_{\al, \beta}(x,(\tan \al) H^0(x))} \times \\ 
	&\hspace{14em} \left\{ (\tan \alpha) H(\nabla_x\phi_{\al, \beta})(\nabla H)(\nabla_x\phi_{\al, \beta}) \cdot \nabla H^0(x) - (\phi_{\al, \beta})_t \right\} dx.
\end{align*}
Now, we compute 
\[
	\frac{1}{\phi_{\al,\beta}} \left\{ (\tan \alpha) H(\nabla_x\phi_{\al,\beta})(\nabla H)(\nabla_x\phi_{\al,\beta}) \cdot \nabla H^0(x) - (\phi_{\al,\beta})_t \right\}
\]
on $\pd C_{\al}$.
Since $\nabla_x \phi_{\al, \beta}(x,t) = A(x,t) \nabla H^0(x)$ where
\begin{align*}
	A(x,t) &= (\Phi^0(x,t))^{\frac{-N-2}{2}} \(\frac{2-N}{2}\) w_{\al, \beta}\(\frac{t^2}{(H^0(x))^2 + t^2} \) H^0(x) \\
	&+ (\Phi^0(x,t))^{-\frac{N-6}{2}} w_{\al, \beta}^{\prime}\(\frac{t^2}{(H^0(x))^2 + t^2}\) (-2t^2 H^0(x)),
\end{align*}
we check that
\[
	H(\nabla_x\phi_{\al, \beta})(\nabla H)(\nabla_x\phi_{\al, \beta}) \cdot \nabla H^0(x) = A(x,t).
\]
Also we have $\frac{t^2}{(H^0(x))^2 + t^2} = \sin^2 \alpha$ on the surface $\pd C_{\alpha}$.
Thus since $t = (\tan \al) H^0(x)$ and $\Phi^0(x,t) = \sqrt{1 + \tan^2 \al} H^0(x)$ on $\pd C_{\al}$, 
we have
\begin{align*}
	&\frac{1}{\phi_{\al, \beta}} \left\{ (\tan \alpha) H(\nabla_x\phi_{\al, \beta})(\nabla H)(\nabla_x\phi_{\al, \beta}) \cdot \nabla H^0(x) - (\phi_{\al, \beta})_t \right\} 
	= \frac{1}{\phi_{\al, \beta}} \left\{ (\tan \alpha) A(x,t) - (\phi_{\al, \beta})_t \right\} \\
	&= \frac{1}{\phi_{\al, \beta}} \left\{ (\Phi^0(x,t))^{\frac{-N-2}{2}} \( \frac{2-N}{2} \) w_{\al, \beta}(\sin^2 \alpha) (\tan \al) H^0(x) \right. \\ 
	&\left. \hspace{2em} + (\Phi^0(x,t))^{\frac{-N-6}{2}} w_{\al, \beta}^{\prime}(\sin^2 \alpha) (-2t^2 (\tan \al) H^0(x)) \right. \\
	&\left. \hspace{2em} - \( \frac{2-N}{2} \) (\Phi^0(x,t))^{\frac{-N-2}{2}} w_{\al, \beta} (\sin^2 \alpha) t 
	- (\Phi^0(x,t))^{\frac{-N-6}{2}} w_{\al, \beta}^{\prime}(\sin^2 \alpha) (2t(H^0(x))^2) \right\} \\
	& = \frac{K(N, \al, \beta)}{\Phi^0(x,t)} \sqrt{1 + \tan^2 \al}
\end{align*}
on $\pd C_{\al}$, where we put
\[
	K(N, \al, \beta) = -(\sin 2\al) w_{\al, \beta}^{\prime}(\sin^2 \alpha) = -(\sin 2\al) \frac{w^{\prime}(\sin^2 \alpha)}{A_{\al, \beta}}.
\]
Summarizing, we have
\begin{align}
\label{div_S2}
	&\int_{\Sigma_{C_{\alpha},2}}\langle \F_{\al, \beta}, \nu_2 \rangle_{N+1} d{\mathcal H}^N \\
	&= K(N, \al, \beta) \sqrt{1 + \tan^2 \al} \int_{\{ x \in \re^{N-1} \, : \, r<\Phi^0(x,(\tan \al) H^0(x))<R \}} 
	\frac{u^2(x,(\tan \al) H^0(x))}{\Phi^0(x,(\tan \al) H^0(x))} dx \notag \\ 
	&= K(N, \al, \beta) \sqrt{1 + \tan^2 \al} 
	\int_{\{ (x, t) \in \pd C_{\al} \, : \, r < \Phi^0(x,t) < R \}} \frac{1}{\sqrt{1 + (\tan^2 \al)|\nabla H^0(x)|^2}} \frac{u^2(x,t)}{\Phi^0(x,t)} d\sigma_{x,t}. \notag
\end{align}

On the other hand, since $\F_{\al, \beta} = \F$, we obtain the estimates
\begin{align}
\label{div_S1}
	&\int_{\Sigma_{C_{\alpha},1}}\langle \F_{\al, \beta}, \nu_1 \rangle_{N+1} d{\mathcal H}^N  = 0, \\
\label{div_S3}
	&\int_{\Sigma_{C_{\alpha},3}}\langle \F_{\al, \beta}, \nu_3 \rangle_{N+1} d{\mathcal H}^N  = O(r^{N-2}), \quad (r \to 0) \\
\label{div_S4}
	&-\int_{\Sigma_{C_{\alpha},4}}\langle \F_{\al, \beta}, \nu_4 \rangle_{N+1} d{\mathcal H}^N 
	\le
	\int_{z \in C_{\al}, r<\Phi^0(z)<R} \( \Phi(\nabla u)^2  - \frac {(\beta-2)^2}4 \frac {u^2(z)}{[\Phi^0(z)]^2} \)dz
\end{align}
as in the case when $\al = 0$, where $\nu_i$ denotes the outer unit normal of $\Sigma_{C_{\alpha},i}$ $(i=1,3,4)$.
Collecting \eqref{div_S2}, \eqref{div_S1}, \eqref{div_S3}, \eqref{div_S4} and
\[
	\sum_{i=1}^4 \int_{\Sigma_{C_{\al},i}} \langle \F_{\al, \beta} \,, \nu_i \rangle_{N+1} d\mathcal H^N=0,
\] 
we obtain the conclusion as in \S \ref{section:proof}.
\qed

\end{document}